\documentclass[12pt,a4paper,reqno]{amsart}
\pdfoutput=1
\usepackage{amsfonts}
\usepackage{amsthm}
\usepackage{amsmath}
\usepackage{amssymb}
\usepackage{mathabx}
\usepackage{bm}
\usepackage{amscd}
\usepackage[latin2]{inputenc}
\usepackage{t1enc}
\usepackage[mathscr]{eucal}
\usepackage{indentfirst}
\usepackage{graphicx}
\usepackage{graphics}
\usepackage{pict2e}
\usepackage{epic}
\numberwithin{equation}{section}
\usepackage[margin=2.9cm]{geometry}
\usepackage{epstopdf} 
\usepackage[colorlinks,linkcolor=blue]{hyperref}
\usepackage{color}
\usepackage{todonotes}
\usepackage[capitalise,noabbrev]{cleveref}

\crefformat{equation}{(#2#1#3)}
\crefrangeformat{equation}{(#3#1#4) to~(#5#2#6)}

\usepackage[normalem]{ulem}

\allowdisplaybreaks

\theoremstyle{plain}
\newtheorem{Thm}{Theorem}[section]
\newtheorem*{Thm*}{Theorem}
\newtheorem{Lem}[Thm]{Lemma}

\newtheorem{Prop}[Thm]{Proposition}
\theoremstyle{definition}

\newtheorem{Rem}[Thm]{Remark}
\newtheorem{?}[Thm]{Problem}

\newcommand{\p}{\partial}

\newcommand{\R}{\mathbb{R}}

\newcommand{\ovl}{\overline}
\newcommand{\e}{\varepsilon}

\newcommand{\ovlul}{\ovl{u}_l}
\newcommand{\ovlur}{\ovl{u}_r}

\begin{document}


\title[Two different perturbations at infinities]{
On Riemann solutions under different initial periodic perturbations at two infinities for 1-d scalar convex conservation laws
}

%
%
\author[Q. Yuan]{Qian YUAN}
\address[Q. Yuan]{The Institute of Mathematical Sciences \&  Department of Mathematics, The Chinese University of Hong Kong,	Shatin, N.T., Hong Kong}
\email{qyuan103@link.cuhk.edu.hk}

\author[Y. Yuan]{Yuan YUAN}
\thanks{The  research of Yuan YUAN is supported by the Start-up Research Grant of South China Normal University, 8S0328.}
\address[Y. Yuan]{South China Research Center for Applied Mathematics and Interdisciplinary Studies, South China Normal University,	Guangzhou, Guangdong, China}
\email{yyuan2102@m.scnu.edu.cn}

\subjclass[2010]{35L03, 35L65, 35L67}
\keywords{conservation laws, shock waves, rarefaction waves, periodic perturbations} 

\begin{abstract} 
	This paper is concerned with the large time behaviors of the entropy solutions to one-dimensional scalar convex conservation laws, of which the initial data are assumed to approach two arbitrary $ L^\infty $ periodic functions as $ x\rightarrow-\infty $ and $ x\rightarrow+\infty, $ respectively. 
	We show that the solutions approach the Riemann solutions at algebraic rates as time increases.
	Moreover, a new discovery in this paper is that the difference between the two periodic perturbations at two infinities may generate a constant shift on the background shock wave, which is different from the result in \cite{Xin2019},  	
	where the two periodic perturbations are the same.

\end{abstract}

\maketitle


\section{Introduction} 

In this paper, we consider the Cauchy problem for a convex scalar conservation law in one-dimensional spatial space:
\begin{align}
\p_t u(x,t)+\p_x f(u(x,t)) & = 0, \quad x \in \R, t>0, \label{equ1} \\
u(x,0) & = u_0(x), \quad x\in\R \label{ic0}
\end{align}
where the flux $f(u)\in C^2(\R)$ is strictly convex, i.e. $f''(u)>0$. 

It is well-known that for any $ L^\infty(\R) $ data $ u_0(x), $ the Cauchy problem \cref{equ1},\cref{ic0} admits a unique entropy solution $ u(x,t), $ satisfying the following entropy condition (see \cite{Oleinik1957}),
\begin{equation}\label{entropy}
u(x+a,t) - u(x,t) \leq \frac{Ea}{t} \quad \text{ a.e. } x\in\R, \forall t>0, a>0,
\end{equation}
where $ E>0 $ is a constant, depending only on $ f $ and $ u_0. $
When the initial data are the Riemann data: 
\begin{equation}\label{ic-Riemann}
u_0(x) = \begin{cases}
\ovlul, & \quad x<0, \\
\ovlur, & \quad x>0,
\end{cases}
\end{equation}
where $ \ovlul $ and $ \ovlur $ are two constants,
the unique entropy solution to \cref{equ1},\cref{ic0} is the well-known Riemann solutions:
\begin{itemize}
	\item[(1)] a shock wave $ u^S(x,t) $ if $ \ovlul > \ovlur, $ where 
	\begin{equation}\label{shock}
	u^S(x,t) := \begin{cases}
	\ovlul, & \quad x<st,\\
	\ovlur, & \quad x>st,
	\end{cases}
	\end{equation} 
	where $ s = \frac{f(\ovlul)-f(\ovlur)}{\ovlul-\ovlur} $ is the shock speed;
	
	\item[(2)] a centered rarefaction wave $ u^R(x,t) $ if $ \ovlul<\ovlur, $ where
	\begin{equation}
	u^R(x,t) := \begin{cases}
	\ovlul, & x\leq f'(\ovlul)t,\\
	(f')^{-1}(\frac{x}{t}), & f'(\ovlul)t < x\leq f'(\ovlur)t, \\
	\ovlur, & x>f'(\ovlur)t,
	\end{cases}
	\end{equation}
	which is Lipschitz continuous when $ t>0; $
	
	\item[(3)] a constant $ \ovlul $ if $ \ovlul = \ovlur. $
\end{itemize}


Riemann solutions are of great importance in the theories of conservation laws. They are the simplest entropy weak solutions, and can be applied as the building blocks to construct global solutions to conservation laws with more general initial data (see \cite{Glimm1965}). Moreover, they also govern the large time behaviors of entropy solutions to conservation laws with the initial data approaching constant states as $ |x|\rightarrow \infty. $ In fact, there have been numerous articles showing that the Riemann solutions are asymptotically stable under localized (for instance, compactly supported) perturbations; see Hopf \cite{Hopf1950}, Lax \cite{Lax1957b}, Ilin-Oleinik \cite{Ilin1960} and Liu \cite{Liu1978} for instance.


For the initial data oscillating at infinity, such as the periodic data, Lax \cite{Lax1957b}, Glimm-Lax \cite{GD} and Dafermos \cite{Dafe1} proved that the entropy solution tends to the constant average in the $ L^\infty $ norm at rates $ 1/t. $ And for the initial data which is a periodic perturbation around a shock wave or a rarefaction wave: 
\begin{equation}\label{ic2}
u_0(x) =
\begin{cases}
\ovl{u}_l+w_{0}(x), & \quad x<0,\\
\ovl{u}_r+w_{0}(x), & \quad x>0,
\end{cases}		
\end{equation}
where $ \ovlul\neq \ovlur $ and $w_0$ is an arbitrary bounded periodic function with zero average, 
Xin-Yuan-Yuan \cite{Xin2019} showed that the entropy solution to \cref{equ1}, \cref{ic2} tends to the background wave in the $ L^\infty $ norm at rates $ 1/t. $
We also refer readers to \cite{XYY2019} for the $ L^\infty $ stability of viscous shock profiles and rarefaction waves under periodic perturbations for the viscous conservation laws.

\vspace{0.3cm}

In this paper, we are concerned about more complicated initial data than that in \cite{Xin2019}, i.e.,  $ u_0 $ coincides two different periodic functions, especially maybe with two different periods, on the two sides of $x$-axis:
\begin{equation}\label{ic1}
u_0(x) =
\begin{cases}
\ovl{u}_l+w_{0l}(x), & \quad x<-N,\\
\ovl{u}_r+w_{0r}(x),& \quad x>N,
\end{cases}		
\end{equation}
where $ N>0 $ is a constant,  and
$w_{0l}(x) \in L^{\infty}(\R)$ and $ w_{0r}(x)\in L^{\infty}(\R)$ are two arbitrary bounded periodic functions with periods $p_l>0$ and $p_r>0$, respectively. 
By adding the averages of $ w_{0l} $ and $ w_{0r} $ on $ \ovlul $ and $ \ovlur, $ respectively, we can assume without loss of generality that both averages are zero, i.e., $ w_{0l} $ and $ w_{0r} $ satisfy
\begin{equation}\label{zero-average}
\frac{1}{p_l} \int_0^{p_l} w_{0l}(x) dx = \frac{1}{p_r} \int_0^{p_r} w_{0r}(x) dx = 0.
\end{equation}
\vspace{0.2cm}


It fails to apply the analysis in \cite{Xin2019}, which is concerned with the initial data \cref{ic2},  directly to the case for two different periodic perturbations \cref{ic1}. For instance, when $ \ovlul>\ovlur, $
a key lemma (Lemma 3.9) in \cite{Xin2019} shows that the entropy solution to \cref{equ1}, \cref{ic2} coincides the two periodic solutions induced by the initial values $\ovlul+w_0$ and $\ovlur+w_0$ on the two sides of a generalized forward characteristic, respectively, 
where these two periodic solutions have very nice properties --- owning divides and decaying fast  to the constant averages as time increases. 
However, for the problem \cref{equ1},\cref{ic1}, 
\cite[Lemma 3.9]{Xin2019} shows that the entropy solution coincides with the two solutions induced by the initial data $ \ovlul+(u_0-\ovlul)\chi_{(-\infty, 0)}+(u_0-\ovlur)\chi_{(0,\infty)} $ and $ \ovlur+(u_0-\ovlul)\chi_{(-\infty, 0)}+(u_0-\ovlur)\chi_{(0,\infty)}, $ respectively, 
which are not periodic any more.
Fortunately, by making use of a comparison theorem \cite[Theorem 11.4.1]{Dafe}, it is found that there are two ``well controlled'' Lipschitz continuous curves $ X_1^*(t) $ and $ X_2^*(t), $  such that $u$ coincides with the periodic solutions induced by the initial data $\ovlul+w_{0l}$ and $\ovlur+w_{0r}$ as $ x<X_1^*(t) $ and $ x>X_2^*(t), $ 
respectively;  see \cref{Prop-gluing}. Although $ X_1^*(t) $ and $ X_2^*(t) $ may not be generalized characteristics of $u,$ they can be well controlled in the sense that for the shock wave ($\ovlul>\ovlur$), they coincide after a finite time; for the rarefaction wave and the constant ($\ovlul\leq\ovlur$), 
both the distances $ |X_1^*(t) - f'(\ovlul)t| $ and $ |X_2^*(t) - f'(\ovlur)t| $ can be bounded by $ O(1)t^{1/2}, $ with the aid of two time-invariant variables; see Lemmas \ref{Lem-shock} and \ref{Lem-X*} for details.
The proof relies on a similar, but more complicated analysis as in \cite{Liu1978}.

In this paper, we prove that if $ \ovlul\leq \ovlur, $ the entropy solution to \cref{equ1}, \cref{ic1} tends to the background Riemann solution in the $ L^\infty $ norm. While if $ \ovlul>\ovlur, $ the asymptotic behavior of the solution is the background shock wave with a constant shift, which is due to the non-zero mass of the compact part of the perturbation and the difference between the periodic perturbations $ w_{0l} $ and $ w_{0r} $ at two infinities; see \cref{shift}. 
It is noted that the result in this paper is compatible with the previous results in Liu \cite{Liu1978} with only compact perturbations and Xin-Yuan-Yuan \cite{Xin2019} with the initial data \cref{ic2}. 

\vspace{0.3cm}

This paper is organized as follows. In Section 2 we first present some notations and the theorems. In Section 3, the decay properties of periodic solutions, the properties of the generalized characteristics and some comparison results are shown. Then the stabilities of shock waves, rarefaction waves and the constants are proved in the subsequent sections. 

\vspace{0.5cm}

\section{Main results}
Let $u_l(x,t)$ and $u_r(x,t)$ denote two periodic entropy solutions to \cref{equ1} with the following initial data
\begin{equation}
u_l(x,0) = \ovl{u}_l+w_{0l}(x) ~\text{ and }~ u_r(x,0) = \ovl{u}_r+w_{0r}(x),~ \text{respectively.}
\end{equation}
For any $u,v\in \R,$ let
\begin{equation}\label{Def-lam-sig}
\sigma(u,v):= \int_0^1 f'(u+\theta (v-u)) d\theta.
\end{equation}
Therefore, $\sigma(\ovlul,\ovlur)$ is the shock speed of the shock $ u^S, $  and if $ u\geq v, $ it holds that
\begin{equation}
\sigma(u,v) \leq \sigma (u,s) \qquad \text{for}~\forall s\in[v,u],                      
\end{equation}
\begin{equation}
f'(v) =\sigma(v,v)\leq \sigma(u,v)\leq \sigma(u,u)=f'(u).
\end{equation}

The main theorems of this paper are presented as follows:
\begin{Thm}\label{Thm-shock}
	Suppose that $ \ovlul>\ovlur $ in \eqref{ic1}. Then for any periodic perturbations $ w_{0l} \in L^\infty(\R) $ and $ w_{0r} \in L^\infty(\R) $ satisfying \eqref{zero-average}, there exist a finite time $ T_S>0 $ and a unique Lipschitz continuous curve $ X(t)\in \text{Lip}(T_S, +\infty), $ such that the entropy solution $ u(x,t) $ to \eqref{equ1}, \eqref{ic1} satisfies that for all $ t>T_S, $
	\begin{equation}
	u(x,t) = 
	\begin{cases}
	u_l(x,t) & \text{ if } x<X(t), \\
	u_r(x,t) & \text{ if } x>X(t).
	\end{cases}
	\end{equation}
	Moreover, there exists a constant $ C>0, $ independent of time $ t, $ such that
	\begin{equation}\label{lb}
	\sup_{x<X(t)} |u(x,t)-\ovl{u}_l| +\sup_{x>X(t)} |u(x,t)-\ovl{u}_r| + |X(t)-st-X_\infty| \leq \frac{C}{t}, \qquad t>0,
	\end{equation}
	with the constant shift $ X_\infty $ given by
	\begin{equation}\label{shift}
	\begin{aligned}
	X_\infty := \frac{1}{\ovlul-\ovlur} \Big[ & \int_{-\infty}^{0} \left( u_0(y)-\ovlul-w_{0l}(y)\right)  dy + \int_{0}^{+\infty} \left( u_0(y)-\ovlur-w_{0r}(y)\right)  dy \\
	& - \min\limits_{x\in\R} \int^x_0 w_{0l}(y) dy + \min\limits_{x\in\R} \int^x_0 w_{0r}(y) dy \Big].
	\end{aligned}
	\end{equation}
\end{Thm}

\begin{Thm}\label{Thm-rare}
	Suppose that $ \ovlul < \ovlur $ in \cref{ic1}. Then for any periodic perturbations $ w_{0l} \in L^\infty(\R) $ and $ w_{0r} \in L^\infty(\R) $ satisfying \eqref{zero-average}, the entropy solution $ u(x,t) $ to \cref{equ1},\cref{ic1} satisfies that
	\begin{equation}\label{result-rare}
	\sup_{x\in\R} |u(x,t) - u^R(x,t)|\leq \frac{C}{\sqrt{t}}, \quad t>0,
	\end{equation}
	where $ C>0 $ is a constant independent of time $ t. $
\end{Thm}

\begin{Thm}\label{Thm-const}
	Suppose that $ \ovlul = \ovlur $ in \cref{ic1}. Then for any periodic perturbations $ w_{0l} \in L^\infty(\R) $ and $ w_{0r} \in L^\infty(\R) $ satisfying \eqref{zero-average}, the entropy solution $ u(x,t) $ to \cref{equ1},\cref{ic1} satisfies that
	\begin{equation}\label{result-const}
	\sup_{x\in\R} |u(x,t) - \ovlul|\leq \frac{C}{\sqrt{t}}, \quad t>0,
	\end{equation}
	where $ C>0 $ is a constant independent of time $ t. $
\end{Thm}
	
%


Compared with the previous result \cite{Xin2019}, which is concerned with the same periodic perturbation on the two sides of $x-$axis, the shift $- \min\limits_{x\in\R} \int^x_0 w_{0l}(y) dy + \min\limits_{x\in\R} \int^x_0 w_{0r}(y) dy $ is new in this paper,
which is essentially due to the difference between these two masses. Thus, even if two periodic perturbations have same periods, the shift may be non-zero in general.


\vspace{0.5cm}

\section{Preliminaries}

Before proving the theorem, in this section we present some preliminary lemmas, and also important propositions derived from the comparison principles of conservation laws.

It is well-known in the theory of generalized characteristics that for one-dimensional scalar convex conservation laws, through every point $(x,t)\in \R \times[0,+\infty)$ pass two forward and two backward extreme generalized characteristics. Moreover, when $t>0,$ the forward generalized characteristic is unique and the two backward extreme generalized characteristics are straight lines.
We refer to \cite[Chapters 10 \& 11]{Dafe} for details, or \cite[Definition 3.1 \& Lemma 3.2]{Xin2019} for a short summary.

\begin{Lem}\label{Lem-per}
	Assume that the initial data $ u_0 \in L^\infty(\R) $ in \cref{ic0} is periodic with period $ p>0 $ and average $ \ovl{u} := \frac{1}{p} \int_{0}^{p} u_0(x)dx. $ Then the entropy solution $ u(x,t) $ to \cref{equ1}, \cref{ic0} is also periodic with period $ p>0 $ and average $ \ovl{u} $ for each $ t\geq 0, $ and satisfies that
	\begin{equation}
		\| u(\cdot,t) - \ovl{u} \|_{L^\infty(\R)} \leq \frac{C}{t}, \quad t>0,
	\end{equation}
	where $ C>0 $ is a constant, independent of time $ t. $
\end{Lem}

\begin{Lem}\cite[Proposition 3.2]{Dafe1}\label{Lem-div}
	Assume that the initial data $ u_0 \in L^\infty(\R) $ in \cref{ic0}. Then the entropy solution $ u(x,t) $ to \cref{equ1}, \cref{ic0} takes a constant value $ \ovl{u} $ along the straight line $ x = f'(\ovl{u}) t + x_0, $ if and only if 
	\begin{equation}
	\int_{x_0}^x (u_0(y)-\ovl{u}) dy \geq 0 \quad \forall x\in\R.
	\end{equation}
\end{Lem}
Such a straight line is a characteristic and it is called a divide (\cite[Definition 10.3.3]{Dafe}).


\begin{Lem}\cite[Theorem 11.8.1]{Dafe}\label{Lem-Dafe}
	Let $u$ and $\tilde{u}$ be entropy solutions to \eqref{equ1} with the respective $ L^\infty $ initial data $u_0$ and $ \tilde{u}_0, $ satisfying 
	$$
	u_0(x)\leq \tilde{u}_0(x) \quad \text{for all~}x\in (y,\tilde{y}).
	$$
	Let $\psi(\cdot)$ be a forward characteristic, associated with the solution $u$, issuing from the point $(y,0)$, and $\tilde{\psi}(\cdot)$ be a forward characteristic, associated with the solution $\ovl{u}$, issuing from the point $(\tilde{y},0)$. Then for any $t>0$ with $\psi(t)<\tilde{\psi}(t)$,
	$$
	u(x,t)\leq \tilde{u}(x,t) \quad \forall x\in (\psi(t),\tilde{\psi}(t)).
	$$
\end{Lem}
\vspace{0.5cm}

Now we first show a comparison principle for the problem \cref{equ1},\cref{ic0}; see \cref{Lem-compare}. And  
two important propositions Propositions \ref{Prop-gluing} and \ref{Prop-invariant} can be proved by using the properties in \cref{Lem-compare}. 

For the initial data \cref{ic1}, one can choose $ z_l \in [0,p_l) $ and $ z_r \in [0,p_r) $ such that 
\begin{equation}\label{points-min}
\int_{0}^{z_l} w_{0l}(x) dx = \min_{x\in \R} \int_{0}^{x} w_{0l}(x) dx, \quad \int_{0}^{z_r} w_{0r}(x) dx = \min_{x\in \R} \int_{0}^{x} w_{0r}(x) dx.
\end{equation}
Hence, for all $ x\in\R, $
\begin{equation}\label{non-negative-1}
\int_{z_l}^{x} w_{0l}(y) dy \geq 0 \quad \text{ and } \quad \int_{z_r}^{x} w_{0r}(y) dy \geq 0.
\end{equation}
Then for each integer $ k\geq 0, $ we define
\begin{equation}\label{Def-gamma}
\Gamma_l^k(t) := z_l-kp_l + f'(\ovlul) t \quad \text{ and } \quad \Gamma_r^k(t) := z_r + kp_r + f'(\ovlur) t.
\end{equation}
It follows from \cref{Lem-div} that $ \Gamma_l^k $ and $ \Gamma_r^k $ are classical characteristics of $ u_l $ and $ u_r, $ respectively, and there hold that
\begin{align}
u_l(\Gamma_l^k(t),t) \equiv \ovlul &, \quad u_r(\Gamma_r^k(t),t) \equiv \ovlur \quad &\forall t\geq 0; \label{constant-div} \\
\int_{\Gamma_l^k(t)}^{x} (u_l(y,t)-\ovlul) dy \geq 0 &, \quad \int_{\Gamma_r^k(t)}^{x} (u_r(y,t)-\ovlur) dy \geq 0 \quad &\forall x\in\R, t\geq 0. \label{non-negative-2}
\end{align}

Now, let $ K>0 $ be a fixed large number, such that
\begin{equation}
\Gamma_l^K(0)=z_l-Kp_l \leq -N \quad \text{ and } \quad \Gamma_r^K(0)=z_r+Kp_r \geq N,
\end{equation}
and then denote
\begin{align*}
X_1(t): ~& \text{a forward characteristic of $u$, issuing from $(\Gamma_l^K(0),0),$ } \\
X_2(t): ~& \text{a forward characteristic of $u$, issuing from $(\Gamma_r^K(0),0)$ }.
\end{align*}

\begin{Lem}\label{Lem-compare}
	With the notations introduced above, there hold that
	\begin{enumerate}
		\item $u_l(x,t)\leq u(x,t)$ if $x<X_1(t)$; $u_l(x,t)\geq u(x,t)$ if $x< \Gamma_l^{K}(t)$;
		\item $u(x,t)\leq u_r(x,t)$ if $x>X_2(t)$; $u(x,t)\geq u_r(x,t)$ if $x>\Gamma_r^{K}(t)$.		
	\end{enumerate} 
\end{Lem}
%
\begin{proof}
	We prove only (1), since the proof of (2) is similar.
	For any fixed $(x,t)$ satisfying $x<X_1(t)$, since $u_l$ is bounded, by the finite propagating speed, one can choose $\psi(0)\in (-\infty,\Gamma_l^K(0))$ small enough, such that a forward characteristic $\psi(t)$ of $u_l$, issuing from $(\psi(0),0)$ satisfies $\psi(t)<x; $ see \cref{Fig-forwards}.
	\begin{figure}[htbp!]
		\includegraphics[width=0.45\textwidth]{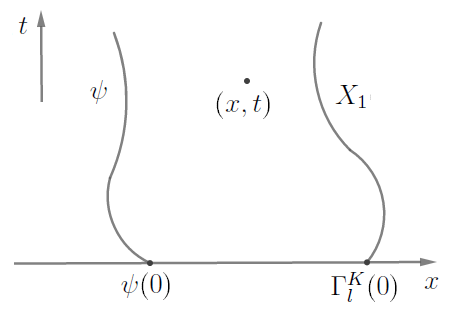} 
		\caption{}\label{Fig-forwards}
	\end{figure}
	It is noted that $u_l(x,0) = u(x,0)$ on $(\psi(0),\Gamma_l^K(0))$ and  $\psi(t)<x<X_1(t)$, then \Cref{Lem-Dafe} implies that $u_l(x,t)\leq u(x,t)$.
	And if $ x< \Gamma_l^K(t), $ it is similar to prove that $ u(x,t) \leq u_l(x,t). $ 
\end{proof}

Now, for any $ t\geq 0, $ we define two Lipschitz continuous curves:
\begin{equation}\label{Def-curves}
X_1^*(t):=\min\{X_1(t), \Gamma_l^K(t) \} \quad \text{ and } \quad X_2^*(t):= \max\{X_2(t), \Gamma_r^K(t)\}.
\end{equation}
Then the following proposition follows from \cref{Lem-compare}.

\begin{Prop}\label{Prop-gluing}
	The unique entropy solution $ u(x,t) $ to \cref{equ1},\cref{ic1} satisfies
	\begin{equation}\label{eq-out}
	u(x,t) = \begin{cases}
	u_l(x,t) \quad \text{ if } x<X_1^*(t), \\
	u_r(x,t) \quad \text{ if } x>X_2^*(t).
	\end{cases}
	\end{equation}
\end{Prop}

\begin{Rem}
	Note that $X_1^*$ and $X_2^*$ may not be the generalized characteristics of $u$. 
    And we will show in the subsequent sections that these carefully chosen curves $X_1^*$ and $X_2^*$ 
    can be ``well-controlled'' at large time.
    It is the large time behaviors of  $X_1^*$ and $X_2^*$, together with \cref{Prop-gluing}, that reveal the main results in this paper.
\end{Rem}


\vspace{0.5cm}

\section{Stability of shock waves}

In this section, the large time behaviors of  $X_1^*$ and $X_2^*$ are studied in the case for $\ovlul>\ovlur$, and then stability of shock waves under two different perturbations at infinities are proved.

\begin{Lem}\label{Lem-shock}
	There exists a finite time $ T_S>0, $ such that 
	$ X_1^*(t) = X_2^*(t) $ for all $ t\geq T_S. $
\end{Lem}
\begin{proof}
It will be established in the following cases.
\begin{enumerate}
	\item[i)] $ \Gamma_l^K(t) \leq X_1(t) $ and $ \Gamma_r^K(t) \geq X_2(t): $ $\Gamma_l^K$ and $\Gamma_r^K$ coincide after some positive time, due to $ \ovlul>\ovlur. $
	
	\vspace{0.3cm}
	
	\item[ii)] $ \Gamma_l^K(t) > X_1(t) $ and $ \Gamma_r^K(t) < X_2(t): $ Define $D(t)=X_2(t)-X_1(t)$, then 
	\begin{align*}
	D'(t) & = \sigma(u(X_2(t)-,t),u(X_2(t)+,t))- \sigma(u(X_1(t)-,t),u(X_1(t)+,t)) \\
	& = \sigma(u(X_2(t)-,t),u_r(X_2(t)+,t))- \sigma(u_l(X_1(t)-,t),u(X_1(t)+,t)).
	\end{align*}
	Now we claim that there exists a finite time $ T_0>0 $ and $ \theta_0 \in (0,1), $ independent of time $ t, $ such that
	\begin{equation}\label{ineq-D}
	D'(t) \leq ~\frac{\theta_0 D(t)}{t} + \frac{1-\theta_0}{2} [ f'(\ovlur) - f'(\ovlul) ], \quad t>T_0.
	\end{equation}
	In fact, define a function 
	\begin{equation}
	h(a,b,\theta) := \sigma(a,b) - \theta f'(a) - (1-\theta) f'(b) , \quad a, b \in \R, \theta \in [0,1].
	\end{equation}
	For any $ \e>0, $ denote a compact domain
	\begin{equation}
	\Omega_\e = \left\lbrace (a,b) \in \R^2: | a- b|\geq \e \text{ and } a^2+b^2 \leq 2 M^2 \right\rbrace,
	\end{equation}
	where $ M:= \max \{ \|u\|_{L^\infty},\| u_l\|_{L^\infty},\|u_r\|_{L^\infty} \}. $ Then for each $ (a,b) \in \Omega_\e, $ the strict convexity of $ f $ implies that there exists a unique $ \theta(a,b) \in (0,1), $ such that $ h(a,b,\theta(a,b)) = 0. $ It follows from the implicit theorem and finite covering theorem that there exist two constants $ 0< \underline{\theta}(\e) \leq \ovl{\theta}(\e) <1, $ such that for each $ (a,b) \in \Omega_\e, $ $ \underline{\theta}(\e) \leq \theta(a,b) \leq \ovl{\theta}(\e). $ 
	
	Now denote $ a_1 = u(X_1(t)+,t) $ and $ b_1 = u_l(X_1(t)-,t), $ then the entropy condition yields $ a_1 \leq b_1. $ If $ (a_1,b_1) \in \Omega_\e, $ then one has that
	\begin{equation}
	\begin{aligned}
	\sigma(a_1,b_1) & = \theta(a_1,b_1) f'(a_1) + (1- \theta(a_1,b_1)) f'(b_1) \\
	& \geq \ovl{\theta}(\e) f'(a_1) + (1-\ovl{\theta}(\e)) f'(b_1).
	\end{aligned}
	\end{equation}
	And if $ (a_1,b_1) \notin \Omega_\e, $ which means $ a_1 > b_1-\e, $ then $ f'(a_1) >f'(b_1) -C\e, $ where $ C>0 $ is a constant, independent of $ \e. $ Then one has that 
	\begin{equation}
	\sigma(a_1,b_1) \geq f'(a_1) \geq \ovl{\theta}(\e) f'(a_1) + (1-\ovl{\theta}(\e)) (f'(b_1) - C\e).
	\end{equation}
	Similarly, for $ a_2 = u(X_2(t)-,t) $ and $ b_2 =  u_r(X_2(t)+,t) $ with $ a_2 \geq b_2, $ one has that if $ (a_2,b_2) \in \Omega_\e, $ 
	$$ \sigma(a_2,b_2) \leq \ovl{\theta}(\e) f'(a_2) + (1-\ovl{\theta}(\e)) f'(b_2); $$ 
	and if $ (a_2,b_2) \notin \Omega_\e, $ i.e. $ a_2 < b_2+\e, $ then
	$$ \sigma(a_2,b_2) \leq \ovl{\theta}(\e) f'(a_2) + (1-\ovl{\theta}(\e)) ( f'(b_2) + C\e). $$ 
	To conclude, one has that 
	\begin{equation}
	\begin{aligned}
	D'(t) = &~ \sigma(a_2,b_2) - \sigma(a_1,b_1) \\
	\leq &~ \ovl{\theta}(\e) (f'(a_2)-f'(a_1)) + (1-\ovl{\theta}(\e)) (f'(b_2) - f'(b_1) + 2C\e) \\
	\leq &~ \ovl{\theta}(\e) \left( \frac{X_2(t) - x_2^0}{t} - \frac{X_1(t) - x_1^0}{t}\right) \\
	& + (1-\ovl{\theta}(\e)) \left( f'(\ovlur) - f'(\ovlul) + \frac{C_1}{t} + 2C\e \right),
	\end{aligned}
	\end{equation}
	where $ x_1^0 $ (resp. $ x_2^0 $) is the intersection of the maximal (resp. minimal) backward characteristic of $ u $ emanating from $ (X_1(t),t) $ (resp. $ (X_2(t),t) $) with $ x$-axis. It is noted that $ -N \leq x_1^0 \leq x_2^0 \leq N, $ then there exist a large $ T_0>0 $ and a small $ \e_0>0, $ such that \cref{ineq-D} holds true.
	
	\vspace{0.2cm}
	
	Hence, it follows from \cref{ineq-D} that
	$$
	\frac{d}{dt} \left(t^{-\theta_0}D(t)  \right) \leq \frac{1-\theta_0}{2} t^{-\theta_0} [f'(\ovlur)-f'(\ovlul)].
	$$
	Thus, for all $ t>T_0, $
	\begin{equation}\label{ineq-D-1}
	D(t) \leq \frac{1}{2}[f'(\ovlur)-f'(\ovlul)] t + C(T_0) t^{\theta_0}.
	\end{equation}
	Since $ f'(\ovlur) < f'(\ovlul), $ \cref{ineq-D-1} implies that $ D(t) = 0 $ after some positive time $ T_1>T_0. $


	\vspace{0.3cm}
	
	\item[iii)] $ \Gamma_l^K(t) > X_1(t) $ and $ \Gamma_r^K(t) \geq X_2(t): $	Define $D_r(t)=\Gamma_r^K(t)-X_1(t)$. 
     Similar to the case ii), there exist $\theta \in (0,1) $ and $ T>0, $ such that for all $ t>T, $
	\begin{align*}
	D_r'(t) = & ~f'(\ovlur) - \sigma(u_l(X_1(t)-,t),u(X_1(t)+,t)) \\
	\leq &~\theta \left[ f'(\ovlur) - f'(u(X_1(t)+,t))\right] + \frac{1}{2}(1-\theta) [f'(\ovlur)-f'(\ovlul)] \\
	= &~\theta\left[\frac{\Gamma_r^K(t)-\Gamma_r^K(0)}{t} -\frac{X_1(t)-x_1^0}{t}\right] + \frac{1}{2}(1-\theta) [f'(\ovlur)-f'(\ovlul)] \\
	\leq &~\frac{\theta D_r(t)}{t} + \frac{1}{2}(1-\theta) [f'(\ovlur)-f'(\ovlul)] + \frac{C(N)}{t}.
	\end{align*}
	Therefore, by similar arguments to that in ii), one can obtain that $X_1$ and $X_r$ coincide after some positive time.
	
	\vspace{0.3cm}
	
	\item[iv)] $ \Gamma_l^K(t) \leq X_1(t) $ and $ \Gamma_r^K(t) < X_2(t): $ The proof is same as above.	
\end{enumerate}
\end{proof}

\begin{proof}[Proof of Theorem \ref{Thm-shock}] 
	Denote $ X(t) := X_1^*(t) = X_2^*(t) $ for $ t>T_S. $ Then due to \cref{Prop-gluing} and \cref{Lem-per}, to finish the proof of \cref{Thm-shock}, it remains to prove the convergence of $ X(t) $ as $ t\rightarrow +\infty. $
	
	For any fixed $ t>T_S, $ one can first choose $ K_1>K $ large enough, such that $ X_1^*(\tau) > \Gamma_l^{K_1}(\tau) $ and $ X_2^*(\tau) < \Gamma_r^{K_1}(\tau) $ for all $ \tau\in [0,t]; $ see \cref{integrate}.
	\begin{figure}[htbp!]
		\includegraphics[width=0.45\textwidth]{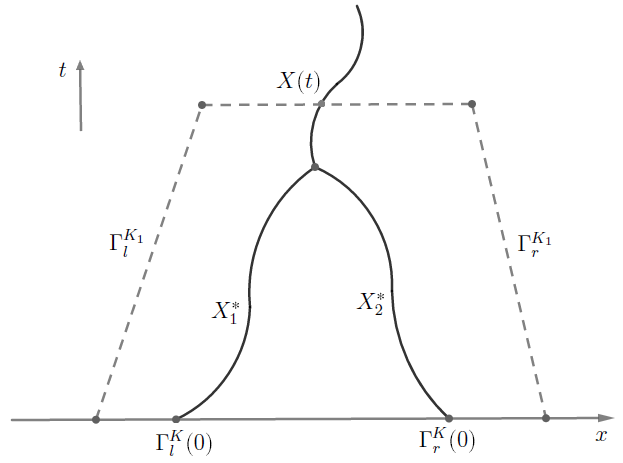}
		\caption{ }
		\label{integrate}
	\end{figure}
	Then integrating the equation \cref{equ1} over the trapezium bounded by $ x $-axis, $ \Gamma_r^{K_1}, $ the line $ \{ (x,t): x\in\R \} $ and $ \Gamma_l^{K_1} $ yields that
	\begin{equation}\label{int-1}
	\begin{aligned}
	& \int_{\Gamma_l^{K_1}(t)}^{\Gamma_r^{K_1}(t)} u(x,t) dx - \int_{\Gamma_l^{K_1}(0)}^{\Gamma_r^{K_1}(0)} u_0(x) dx \\
	= & \int_0^t \left[ f'(\ovlur) u(\Gamma_r^{K_1}(\tau),\tau) - f( u(\Gamma_r^{K_1}(\tau),\tau))\right] d\tau \\
	& - \int_0^t \left[ f'(\ovlul) u(\Gamma_l^{K_1}(\tau),\tau) - f( u(\Gamma_l^{K_1}(\tau),\tau))\right] d\tau \\
	= & ~[ f'(\ovlur) \ovlur - f(\ovlur) ] t - [ f'(\ovlul) \ovlul - f(\ovlul) ] t
	\end{aligned}
	\end{equation}
	where the second equality follows from \cref{Prop-gluing} and \cref{constant-div}. Also, \cref{Prop-gluing} implies that
	\begin{equation}\label{int-2}
	\begin{aligned}
	\int_{\Gamma_l^{K_1}(t)}^{\Gamma_r^{K_1}(t)} u(x,t) dx = & \int_{\Gamma_l^{K_1}(t)}^{X(t)} u_l(x,t) dx + \int_{X(t)}^{\Gamma_r^{K_1}(t)} u_r(x,t) dx \\
	= &~ \ovlul (X(t)-\Gamma_l^{K_1}(t)) + \ovlur (\Gamma_r^{K_1}(t) - X(t)) + O(1) \frac{p_l^2+p_r^2}{t} \\
	= &~ (\ovlul - \ovlur) X(t) - \ovlul (z_l-K_1p_l+f'(\ovlul)t) \\
	& + \ovlur (z_r+K_1p_r+f'(\ovlur)t) + \frac{O(1)}{t},
	\end{aligned}
	\end{equation}
	where the bound in $ O(1) $ depends only on $ f. $
	And \cref{zero-average} yields that
	\begin{align}
	\int_{\Gamma_l^{K_1}(0)}^{\Gamma_r^{K_1}(0)} u_0(x) dx = & \int_{\Gamma_l^{K_1}(0)}^{0} (u_0(x)-\ovlul-w_{0l}(x)) dx - \ovlul \Gamma_l^{K_1}(0) + \int_{z_l}^0 w_{0l}(x) dx \notag \\
	& + \int_0^{\Gamma_r^{K_1}(0)} (u_0(x)-\ovlur-w_{0r}(x)) dx + \ovlur \Gamma_r^{K_1}(0) + \int_0^{z_r} w_{0r}(x) dx \notag \\
	= & (\ovlul - \ovlur) X_\infty - \ovlul(z_l - K_1 p_l) + \ovlur (z_r + K_1 p_r), \label{int-3}
	\end{align}
	where \cref{points-min} is used and $ X_\infty $ is defined in \cref{shift}.
	Then it follows from \cref{int-1,int-2,int-3} that
	\begin{align*}
	(\ovlul - \ovlur) X(t) - (f(\ovlul) - f(\ovlur)) t = (\ovlul - \ovlur) X_\infty + \frac{O(1)}{t},
	\end{align*}
	which finishes the proof of \cref{Thm-shock}.

\end{proof}

\vspace{0.3cm}

\section{Stabilities of rarefaction waves and constants}

In this section, we consider the initial data \cref{ic1} where $ \ovlul \leq \ovlur. $ 
To study the large time behaviors of  $X_1^*$ and $X_2^*$, two time-invariant variables are introduced. 

For any fixed $ t\geq 0, $ we first let $ K_2=K_2(t) > K $ denote a large integer, such that $ \Gamma_l^{K_2}(\tau) < X_1^*(\tau) $ and $ \Gamma_r^{K_2}(\tau) > X_2^*(\tau) $ for all $ \tau \in [0,t]; $ see \cref{Fig-rare}.
\begin{figure}[htbp!]
	\includegraphics[width=0.45\textwidth]{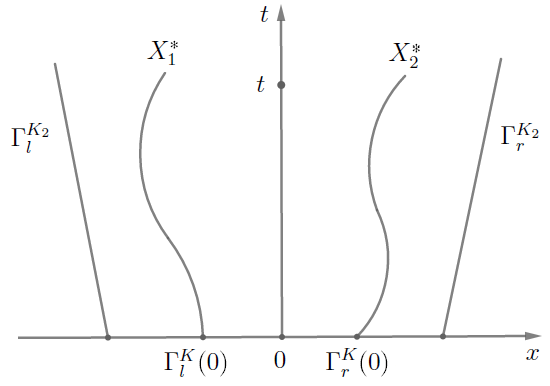}
	\caption{ }
	\label{Fig-rare}
\end{figure}

Then \cref{Prop-gluing} implies that
\begin{align*}
& \int_{\Gamma_l^k(t)}^{x} (u(y,t)-\ovlul) dy = \int_{\Gamma_l^{K_2}(t)}^{x} (u(y,t)-\ovlul) dy \quad \forall k\geq K_2; \\
& \int_{x}^{\Gamma_r^k(t)} (u(y,t)-\ovlur) dy = \int_{x}^{\Gamma_r^{K_2}(t)} (u(y,t)-\ovlur) dy \quad \forall k\geq K_2.
\end{align*}
Moreover, when $ x<\Gamma_l^{K_2}(t) $ (or $ x>\Gamma_r^{K_2}(\tau) $), \cref{Prop-gluing} and \cref{non-negative-1} yield that the integral 
\begin{align*}
& \int_{\Gamma_l^{K_2}(t)}^{x} (u(y,t)-\ovlul) dy = \int_{\Gamma_l^{K_2}(t)}^{x} (u_l(y,t)-\ovlul) dy \geq 0 \\
\Big(\text{or } & \int_{x}^{\Gamma_r^k(t)} (u(y,t)-\ovlur) dy = \int_{x}^{\Gamma_r^k(t)} (u_r(y,t)-\ovlur) dy \leq 0 \Big) 
\end{align*}
is bounded and periodic with respect to $ x. $
Hence, for each $ t\geq 0, $ we can well define
\begin{equation}\label{Def-Inva}
\begin{aligned}
P(t) & := \lim\limits_{k\rightarrow +\infty} \min_{x\in \R} \int_{\Gamma_l^k(t)}^{x} (u(y,t)-\ovlul) dy \\
& = \min_{x\in \R} \int_{\Gamma_l^{K_2}(t)}^{x} (u(y,t)-\ovlul) dy \leq 0; \\
Q(t) & := \lim\limits_{k\rightarrow +\infty} \max_{x\in \R} \int_{x}^{\Gamma_r^k(t)} (u(y,t)-\ovlur) dy \\
& = \max_{x\in \R} \int_{x}^{\Gamma_r^{K_2}(t)} (u(y,t)-\ovlur) dy \geq 0.
\end{aligned}
\end{equation}
It is found that these two functions are actually time-invariant:
\begin{Prop}\label{Prop-invariant}
	There exist two points $ x_P, x_Q \in [\Gamma_l^K(0), \Gamma_r^K(0)], $ such that for all $ t\geq 0, $
	\begin{align}
	P(t) & = \int_{\Gamma_l^{K_2}(t)}^{x_P+f'(\ovlul)t} (u(y,t)-\ovlul) dy = P(0), \label{invariant-P}\\
	Q(t) & = \int_{x_Q+f'(\ovlur)t}^{\Gamma_r^{K_2}(t)} (u(y,t)-\ovlur) dy = Q(0).
	\end{align}
\end{Prop}
\begin{proof}
	We prove only \cref{invariant-P}, since the other one is similar.
	
	By \cref{zero-average}, one has that
	\begin{equation*}
	P(0) = \min_{x\in \R} \int_{\Gamma_l^{K}(0)}^{x} (u_0(y)-\ovlul) dy.
	\end{equation*}
	Then it follows from \cref{non-negative-1} and $ \ovlul \leq \ovlur $ that 
	\begin{align*}
	& \text{for } x\leq \Gamma_l^K(0), \quad \int_{\Gamma_l^{K}(0)}^{x} (u_0(y)-\ovlul) dy = \int_{\Gamma_l^{K}(0)}^{x} w_{0l}(y) dy \geq 0; \\
	& \text{for } x\geq \Gamma_r^K(0), \quad \int_{\Gamma_r^{K}(0)}^{x} (u_0(y)-\ovlul) dy \geq \int_{\Gamma_r^{K}(0)}^{x} w_{0r}(y) dy \geq 0.
	\end{align*}
	Thus	one can find a point $ x_P \in [\Gamma_l^K(0), \Gamma_r^K(0)], $ such that
	\begin{equation*}
	P(0) = \int_{\Gamma_l^K(0)}^{x_P} (u_0(y)-\ovlul) dy,
	\end{equation*}
	which implies that 
	\begin{equation}
	\int_{x_P}^{x} (u_0(y) - \ovlul) dy \geq 0 \quad \forall x\in\R.
	\end{equation}
	This, with \cref{Lem-div} yields that
	\begin{equation}\label{equ-divide}
	u(x_P+f'(\ovlul)t, t) = \ovlul \quad \forall t\geq 0.
	\end{equation}
	Then by applying \cref{Lem-div} once more, \cref{equ-divide} implies that for all $ t>0, $
	\begin{equation}
	\int_{x_P+f'(\ovlul)t}^{x} (u(y,t)-\ovlul) dy \geq 0 \quad \forall x\in\R,
	\end{equation}
	which means that for all $ t\geq 0, $
	\begin{equation}\label{equality-1}
	P(t) = \int_{\Gamma_l^{K_2}(t)}^{x_P+f'(\ovlul)t} (u(y,t)-\ovlul) dy.
	\end{equation}
	Thus, for each fixed $ t>0, $ by integrating the equation \cref{equ1} over the parallelogram $ \{ (y,\tau): 0<\tau<t, ~\Gamma_l^{K_2}(\tau) < y < x_P+f'(\ovlul)\tau \}, $ and applying \cref{Prop-gluing}, \cref{constant-div} and \cref{equ-divide}, one can finish the proof.
\end{proof}

Now, with the two time-invariant variables $ P $ and $ Q, $ we will estimate the distances $ X_1^*(t) - f'(\ovlul) t $ and $ X_2^*(t) - f'(\ovlur) t. $

\begin{Lem} \label{Lem-X*}
	There exists a constant $ C>0, $ independent of time $ t, $ such that for all $ t>0, $
	\begin{align}
	X_1^*(t) \geq f'(\ovlul) t - C\sqrt{t},\label{esti-l} \\
	X_2^*(t) \leq f'(\ovlur) t + C\sqrt{t}.
	\end{align}
\end{Lem}
\begin{proof}
	We prove only \cref{esti-l}, since the other one is similar.
	
	
%

	\begin{itemize}
	\item[i) ] If $ X_1(t) \geq \Gamma_l^K(t): $ By the definition \cref{Def-curves}, \cref{esti-l} holds true obviously.
	
	\item[ii) ] If $ X_1(t) < \Gamma_l^K(t): $ Denote $ h(s) := (f')^{-1}(s). $ It is noted that both $ X_1 $ and the straight line $ x =x_P+f'(\ovlul) t $ (see \cref{equ-divide}) are forward characteristics of $ u, $ then for $ x\in (X_1(t), x_P+f'(\ovlul)t), $
	\begin{align}
	u(x,t)-\ovlul & = h'(\cdot) (f'(u)-f'(\ovlul)) \\
	& = h'(\cdot) (\frac{x-x^0}{t}-f'(\ovlul)),
	\end{align}
	where $ x^0 \in [\Gamma_l^K(0), x_P] $ and $ h'(\cdot) \in [c_1,c_2]. $
	
	Thus 
	\begin{align*}
	P(0) & = P(t) = \int_{\Gamma_l^{K_2}(t)}^{x_P+f'(\ovlul)t} (u(x,t)-\ovlul)dx \\
	& = \int_{\Gamma_l^{K_2}(t)}^{X_1(t)} (u_l(x,t)-\ovlul)dx + \int_{X_1(t)}^{x_P+f'(\ovlul)t} (u(x,t)-\ovlul)dx \\
	& \leq \frac{C}{t} + \int_{X_1(t)}^{x_P+f'(\ovlul)t} h'(\cdot) (\frac{x-\Gamma_l^K(0)}{t}-f'(\ovlul)) dx \\
	& = \int_{a(t)}^{b(t)} h'(\cdot)\frac{x}{t} dx + \frac{C}{t},
	\end{align*} 
	where $ a(t):=X_1(t)-f'(\ovlul)t-\Gamma_l^K(0) $ and $ b(t):=x_P-\Gamma_l^K(0) \in [0, \Gamma_r^K(0)-\Gamma_l^K(0)]. $
	Hence, if $ a(t) \geq 0, $ \cref{esti-l} holds true. And if $ a(t) < 0, $ then
	\begin{align}
	P(0) \leq \frac{c_2}{2t} b^2(t) - \frac{c_1}{2t} a^2(t) + \frac{C}{t}.
	\end{align}
	This, with $ P(0)\leq 0, $ yields that $ a(t) = -|a(t)| \geq -C' \sqrt{t}. $ The proof of the lemma is finished.

	\end{itemize}	
%
%
%
%
	
\end{proof}

Once obtained the large time behaviors of $X_1^*$ and $X_2^*$, the stabilities of rarefaction waves and constants under two different periodic perturbations are easily proved.

\begin{proof}[Proof of \cref{Thm-rare}]
	Assume that $ \ovlul<\ovlur $ in \cref{ic1}. Then for each fixed $ t>0, $ we prove \cref{result-rare} in the following cases.
	
	\begin{itemize}
		\item[i) ] If $ x < X_1^*(t), $ \cref{result-rare} follows from \cref{Prop-gluing} and \cref{Lem-per} immediately.
		
		\item[ii) ] If $ X_1^*(t) < x < f'(\ovlul)t, $ then 
		\begin{equation}\label{ineq-1}
		\frac{x-\Gamma_r^K(0)}{t} \leq f'(u(x,t)) \leq \frac{x-\Gamma_l^K(0)}{t},
		\end{equation}
		and it follows from \cref{Lem-X*} that
		$$ \frac{f'(\ovlul)t - C\sqrt{t}}{t} \leq \frac{X_1^*(t)}{t} \leq \frac{x}{t} \leq f'(\ovlul). $$
		Hence, one can get that
		\begin{equation}
		|f'(u(x,t)) - f'(\ovlul)| \leq \frac{C}{\sqrt{t}},
		\end{equation}
		which implies \cref{result-rare}.
		
		\item[iii) ] If $ f'(\ovlul)t < x < f'(\ovlur)t, $ \cref{ineq-1} still holds true, and then
		\cref{result-rare} follows immediately.
		
		\item[iv) ] The proof of remaining cases are similar to i) and ii).
	\end{itemize}
\end{proof}

\begin{proof}[Proof of \cref{Thm-const}]
	Assume that $ \ovlul=\ovlur $ in \cref{ic1}. Then the proof is similar to that of \cref{Thm-rare}, regardless of the case iii).
	
\end{proof}

\vspace{0.2cm}

Acknowledgments: The authors would like to thank Zhouping Xin for his great suggestions on this topic, and Peng Qu for many helpful discussions.


\newpage
\bibliographystyle{amsplain}
\bibliography{bibli_two}

\providecommand{\bysame}{\leavevmode\hbox to3em{\hrulefill}\thinspace}
\providecommand{\MR}{\relax\ifhmode\unskip\space\fi MR }
\providecommand{\MRhref}[2]{%
  \href{http://www.ams.org/mathscinet-getitem?mr=#1}{#2}
}
\providecommand{\href}[2]{#2}
\begin{thebibliography}{10}

\bibitem{Dafe1}
C.~M. Dafermos, \emph{Large time behavior of periodic solutions of hyperbolic
  systems of conservation laws}, J. Differential Equations \textbf{121} (1995),
  183--202.

\bibitem{Dafe}
\bysame, \emph{Hyperbolic conservation laws in continuum physics}, vol. 325,
  Springer-Verlag, Berlin, 2000.

\bibitem{Glimm1965}
J.~Glimm, \emph{{Solutions in the large for nonlinear hyperbolic systems of
  equations}}, Comm. Pure Appl. Math. \textbf{18} (1965), 697--715.

\bibitem{GD}
J.~Glimm and P.~D. Lax, \emph{Decay of solutions of systems of nonlinear
  hyperbolic conservation laws}, American Mathematical Society, Providence,
  R.I., 1970.

\bibitem{Hopf1950}
E.~Hopf, \emph{The partial differential equation $u_t+uu_x= \mu u_{xx}$}, Comm.
  Pure Appl. Math. \textbf{3} (1950), 201--230.

\bibitem{Ilin1960}
A.~M. Il'in and O.~A. Oleinik, \emph{{Asymptotic behavior of solutions of the
  Cauchy problem for some quasilinear equations for large values of time}},
  Mat. Sb. \textbf{51(93)} (1960), no.~2, 191--216.

\bibitem{Lax1957b}
P.~D. Lax, \emph{{Hyperbolic systems of conservation laws. II}}, Comm. Pure
  Appl. Math. \textbf{10} (1957), 537--566.

\bibitem{Liu1978}
T.-P. Liu, \emph{Invariants and asymptotic behavior of solutions of a
  conservation law}, Proc. Amer. Math. Soc. \textbf{71} (1978), no.~2,
  227--231. \MR{500495}

\bibitem{Oleinik1957}
O.~A. Oleinik, \emph{{Discontinuous solutions of non-linear differential
  equations}}, Uspekhi Mat. Nauk \textbf{12} (1957), no.~3 (75), 3--73.

\bibitem{XYY2019}
Z.~Xin, Q.~Yuan, and Y.~Yuan, \emph{Asymptotic stability of shock profiles and
  rarefaction waves under periodic perturbations for 1-d convex scalar viscous
  conservation laws},  (2019).

\bibitem{Xin2019}
\bysame, \emph{Asymptotic stability of shock waves and rarefaction waves under
  periodic perturbations for 1-d convex scalar conservation laws}, SIAM Journal
  on Mathematical Analysis \textbf{51} (2019), no.~4, 2971--2994.

\end{thebibliography}

\end{document}